\documentclass[a4paper,11pt]{article}
\usepackage[utf8]{inputenc}
\usepackage[english]{babel}
\usepackage[ruled,section]{algorithm}
\usepackage{makeidx}
\usepackage[font=small]{caption}
\usepackage{amsmath}
\usepackage{amsthm}
\usepackage{amsfonts}
\usepackage{amssymb}
\usepackage{mathrsfs}
\usepackage{graphicx}
\usepackage[margin=25mm]{geometry}
\usepackage{enumerate}
\usepackage{indentfirst}
\usepackage[none]{hyphenat}
\usepackage{color}
\usepackage{authblk}
\makeindex
\newcommand{\real}{\mathbb{R}}
\newcommand{\n}{\mathbb{N}}
\newcommand{\rN}{ {\mathbb{R}^N} }

\numberwithin{equation}{section}
\newtheorem{teo}{Theorem}[section]
\newtheorem{lem}[teo]{Lemma}
\newtheorem{rmk}[teo]{Remark}

\newtheorem{defin}[teo]{Definition}
\newtheorem{corol}[teo]{Corollary}
\newtheorem{prop}[teo]{Proposition}

\usepackage[colorlinks=false,backref=true,pagebackref=false,pdftex,pdfauthor={Ederson, Gabrielle},pdftitle={Symmetry properties of positive solutions for fully nonlinear elliptic systems}]{hyperref}

\begin{document}

\title{\bf\Large Symmetry properties of positive solutions for fully nonlinear elliptic systems}
\author[1]{Ederson Moreira dos Santos\footnote{ederson@icmc.usp.br.}}
\author[1]{Gabrielle Nornberg\footnote{gabrielle@icmc.usp.br.}}
\affil[1]{\small Instituto de Ciências Matemáticas e de Computação, Universidade de São Paulo, Brazil}
\date{}

\maketitle

{\small\noindent{\bf{Abstract.}} We investigate symmetry properties of positive  solutions for fully nonlinear uniformly elliptic systems, such as
$$
F_i \,(x,Du_i,D^2u_i) +f_i \,(x,u_1, \ldots , u_n,Du_i)=0, \;\; 1 \leq i \leq n,
$$
in a bounded domain $\Omega$ in $\mathbb{R}^N$ with Dirichlet boundary condition $u_1=\ldots,u_n=0$ on $\partial\Omega$. Here, $f_i $'s are nonincreasing with the radius $r=|x|$, and satisfy a cooperativity assumption. In addition, each $f_i $ is the sum of a locally Lipschitz with a nondecreasing function in the variable $u_i$, and may have superlinear gradient growth. 
We show that symmetry occurs for systems with nondifferentiable $f_i$'s by developing a unified treatment of the classical moving planes method in the spirit of Gidas-Ni-Nirenberg.
We also present different applications of our results, including uniqueness of positive solutions for Lane-Emden systems in the subcritical case in a ball, and symmetry for a class of systems with natural growth in the gradient. 
}
\smallskip

{\small\noindent{\bf{Keywords.}} {Positive solution; Elliptic system of fully nonlinear equations; Symmetry; Moving planes.}

\smallskip

{\small\noindent{\bf{MSC2010.}} {35J47, 35J60, 35N25, 35B06, 35B50.}

\section{Introduction and main results}\label{Introduction}
 
In this paper we study radial symmetry of solutions for fully nonlinear uniformly elliptic systems in the following form
\begin{align}\label{P} \tag{P}
\left\{
\begin{array}{rclcl}
F_i \,(x,Du_i,D^2u_i) +f_i \,(x,u_1, \ldots , u_n,Du_i)&=&0 &\mbox{in} & \;\Omega \\
u_i&>& 0 &\mbox{in} & \;\Omega \\
u_i&=& 0 &\mbox{on} & \partial\Omega, \;\;\, 1 \leq i \leq n,
\end{array}
\right.
\end{align}
where $\Omega$ is a bounded $C^{2}$ domain in $\rN$, $N\geq 1$, $n\geq 1$, and $f_i$ is not necessarily locally Lipschitz. 

Symmetry properties of partial differential equations, for general, are of independent interest since it always makes sense to ask whether or not solutions inherit the same symmetry from the differential operator and from the domain of definition.
On the other hand, special attention has been devoted to radial solutions of specific problems. For instance, the Lane-Emden conjecture has been fully solved in the radial setting \cite[Theorem 3.1]{Mitidieri}, and the uniqueness of positive radial solutions for Lane-Emden systems was proved in \cite[Theorem 1.1(i)]{Dalmasso}. In particular, in these specific cases, it should be natural to reach the full result via symmetry. We adopt this procedure to prove the uniqueness of positive solution to Lane-Emden systems; cf.\ Corollary \ref{uniqLE} ahead.

Radial symmetry has been extensively studied in the literature since the seminal works of Serrin \cite{Serrin} and Gidas-Ni-Nirenberg \cite{GNiN}, which are based upon the \textit{moving planes} method.
The method was revisited in the influential paper \cite{BN} of Berestycki-Nirenberg, in which a central tool to start moving the planes is the maximum principle in small domains. This permitted the authors to remove the original twice differentiability assumption up to the boundary on solutions (and the regularity of the boundary), although they imposed a locally Lipschitz condition which must be true even when $u=0$. 
In the recent years, a lot of variants were considered. For example, assuming differentiability on the $f_i$'s, in \cite{DPacella} the authors considered a different type of symmetry, namely foliated Schwarz, related to solutions having low Morse index. On the non-Lipschitz scenario, it was performed in \cite{DF} a local moving plane method followed by a unique continuation principle. 

The generalization of the pioneering  method for systems was first considered in \cite{Troy}. 
The respective version of \cite{BN} for cooperative systems in the differentiable case can be found in \cite{DjairoSim}, and in the most works that have picked out this approach since then.
It is not our intention here to give a full literature review. Instead we quote other few papers as \cite{BLP, BdaLio, Dancer, Li91, Serra}, and references therein, 
in which a more clear scenario can be built related to both equations and systems.

On the other hand, in contrast to the scalar case, several prototype problems involving systems in the superlinear setting are naturally not differentiable or even not Lipschitz; e.g.\ \eqref{LE} and \eqref{eq.terracini.geral}.
Up to our knowledge, symmetry results for systems without differentiability hypotheses on $f_i$ are not properly available.
The results in this paper include these cases and provide a better understanding about symmetry properties for systems by featuring the essence of Gidas-Ni-Nirenberg technique. 
Consequently, we clarify some divergences that appeared in the literature over the past years since the work \cite{Shaker}; see \eqref{conta f} where the cooperativity assumption ($H_4$) is used. We focus on relaxing Lipschitz or differentiability hypotheses on $f_i$, even under the price of asking natural regularity on the solutions $u_i$ and on the domain.

Next we list our hypotheses.
First and foremost, a consistent requirement over equations when dealing with radial symmetry is their rotational invariance property.
For our operators $F_i$, this will be expressed in terms of an exclusive dependence on the eigenvalues of $D^2 u$ and on the lengths of $Du$ and $x$. 
We assume, as in \cite{BSoverdet}, that $\mathcal{F}_i:\overline{\Omega}\times\real^n\times\rN \times\mathbb{S}^N(\mathbb{R})\rightarrow\real$, where $\mathcal{F}_i(x,\textbf{u},p,X) = F_i(x,p,X) + f_i(x,\textbf{u},p)$, with $\textbf{u}=(u_1, \ldots, u_n)$, $1 \leq i \leq n$, satisfy the following invariance
\smallskip
\begin{itemize}
\item[($H_0$)] $\mathcal{F}_i\,(x,\textbf{u},Qp,Q^t X Q) = \mathcal{F}_i\, (x,\textbf{u},p,X)$ for all $N \times N$ orthogonal matrix $Q$, and for all $x,\textbf{u},p,X$. 
\end{itemize}
\smallskip
Here $\mathbb{S}^N(\real)$ is the subspace of symmetric matrices of order $N$ with real entrances. 
Concerning structure, to apply the moving planes method, as in \cite{BdaLio}, we consider on each $F_i$ the following condition,

\begin{itemize}
\item[($H_1$)]
$\mathcal{M}_i^- (X-Y)-\gamma |p-q| \leq F_i\, (x,p,X) - F_i\, (x,q,Y) \leq \mathcal{M}^+_i (X-Y)+\gamma |p-q|$, for all $X,Y\in \mathbb{S}^N(\real)$, $p,q \in \rN$, and $x\in\Omega$, with $F_i(\cdot,0,0)\equiv 0$,
\end{itemize} 
\smallskip
where $\mathcal{M}_i^\pm$ are the Pucci's extremal operators; see  Section~\ref{Preliminaries}. 
In addition, in order to treat possible superlinear growth in the gradient, we assume that for any $\sigma >0$, there exists $\mu_\sigma \geq 0$ such that
\begin{itemize}
\item[($H_2$)]
$| f_i\, (x,\textbf{u},p) - f_i\, (x,\textbf{u},q) |\leq \mu_\sigma |p-q|\,$ for $p,q\in \overline{B}_\sigma$, $x\in \Omega$ and $\textbf{u}=(u_1,\ldots,u_n)\in \overline{B}_\sigma $.
\end{itemize} 
For instance, we refer to \cite{KoikePerron} where a hypothesis of this nature appears to treat equations with quadratic gradient growth. 
Moreover, for each $i\in \{1,\ldots, n\}$ we assume
\smallskip
\begin{itemize}
\item[($H_3$)] $F_i$ and $f_i$ are radially symmetric, and $F_i+f_i$ is nonincreasing with $r = |x|$, for each fixed $ \textbf{u},p,X$;\smallskip
\item[($H_4$)] $f_i=f_{i,1}+f_{i,2}$, where $f_{i,1}$ is uniformly locally Lipschitz in the component $u_i$, and $f_{i,2}$ is nondecreasing in $u_i$, whenever the remaining components $u_j$, for $j\neq i$, and $x,p$ are fixed;\smallskip
\item[($H_5$)] $ u_j\mapsto f_i \,(x,u_1,\ldots, u_n,p)$ is nondecreasing if $i\neq j$, and $x$, $p$, and $u_l$, for $l\neq j$, are fixed.
\end{itemize}
\smallskip

The latter represents cooperativity for the system \eqref{P}, which ensures the validity of the strong maximum principle; compare with \cite[eq.\ (1.3)]{Troy} in the differentiable setting, and \cite[eq.\ (1.2)]{DM} in the linear one. Besides sufficient, this condition is also necessary in order to preserve the symmetry; see Section~\ref{secao aplicacoes}.

We use the shorthand notation \eqref{H} for the preceding hypotheses, that is,
\begin{align}\label{H}\tag{$H$}
(H_0), \;(H_1), \;(H_2), \; (H_3), \; (H_4),\textrm{ and } (H_5).
\end{align}

In what follows, we say solutions to mean the classical solutions of the Dirichlet problem \eqref{P}, being twice differentiable up to the boundary. We stress that this is always the case if, for instance, each $f_i(x,\textbf{u}(x),Du_i(x))$ is $ C^\alpha (\overline{\Omega}) $ as a function of $x$, and $F_i$ is concave or convex in the $X$ entry (in particular the Pucci's operators); see \cite{CafCab} and \cite{Winter}.
On the other hand, viscosity solutions are understood in the $C$-viscosity sense; see the next section. 

We state our  main results in the sequel. 
\begin{teo}\label{Tbola}
Let $\Omega$ be a ball of radius $R$ centered at the origin.  Assume \eqref{H} for all $1\leq i\leq n$.
Let $(u_1, \ldots , u_n)$  be a solution of \eqref{P}, where $u_i$'s are  $C^2(\overline{\Omega})$ functions. Then $u_i$ is radially symmetric and $\partial_r u_i <0$, for all $r=|x|\in (0,R)$, $1\leq i\leq n$. 
\end{teo}
 
If the operators $F_i$,  $1 \leq i \leq n$, are continuously differentiable in the entry $X \in \mathbb{S}^N(\mathbb{R})$, it is derived the following result concerning overdetermined problems in general smooth domains for our type of systems. 
Consistently with the classical notation \cite{Serrin, BSoverdet}, the involved functions do not depend on $x$.

\begin{teo} \label{Toverd}
Let $\Omega$ be a bounded $C^{2}$ domain.
Assume that $F_i\,(p,X)$ is continuously differentiable in $X$, and that $f_i$ and $F_i$ verify \eqref{H} without dependence on $x$, for all $1\leq i\leq n$.
Let $(u_1, \ldots , u_n)$  be a solution of \eqref{P}, where $u_i$'s are  $C^2(\overline{\Omega})$ functions. Assume that $\partial_{\nu} u_i=c_i$ on $\partial\Omega$, where $c_i$ is a constant and $\nu$ is the unit interior normal to $\partial\Omega$. 
Then $\Omega$ must be a ball, $u_i$ is radially symmetric and $\partial_r u_i <0$, for all $1\leq i\leq n$.  
\end{teo}

As already observed in \cite{BSoverdet}, the $C^1$ hypothesis on $F_i$ is necessary in order to apply Serrin's lemma \cite[Lemmas 1 and 2]{Serrin}, which is a version of Hopf lemma in domains with corners.

A model case for Theorem \ref{Tbola} consists of Pucci's operators, or more generally $\mathcal{F}_i(x,\textbf{u},p,X)=\mathcal{M}_i^\pm (X)\pm \gamma_i (|x|)|p|\pm \mu_i(|x|)|p|^2+f_i(|x|,\textbf{u})$ for some bounded functions $\gamma_i, \mu_i \geq 0$, and we emphasize that our results are new even for systems involving the Laplacian operator. 
For Theorem~\ref{Toverd}, a simple example is  $\mathcal{F}_i(\textbf{u},p,X)= \mathrm{tr} (A^i X)+b_i|p| +\mu_i|p|^2+f_i(\textbf{u})$, where $A^i$ is a positive matrix, and $b_i,\mu_i$ are nonnegative constants. 

Regarding viscosity solutions, it is possible to obtain a restatement of Theorem \ref{Tbola} for the ones which are continuously differentiable up to the boundary. In this case we need to impose continuity of the operators $F_i$ in the variable $x$ as in \cite[(H1)]{BdaLio}, namely
\smallskip
\begin{itemize}
\item[($H_6$)]
$| F_i\, (x,p,X) - F_i\, (y,q,Y) |\leq \gamma\{ |p-q|+\|X-Y\| +|x-y|(\|X\|+\|Y\| ) \}  +\omega (|x-y|(1+|p|+|q|) )$
\end{itemize} 
\smallskip
for all $X,Y\in \mathbb{S}^N(\real)$, $p,q \in \rN$, and $x,y\in\Omega$, for some continuous function $\omega$ with $\omega (0)=0$. 

\begin{teo}\label{Tvisc}
Let $\Omega$ be a ball of radius $R$ centered at the origin.  Assume \eqref{H}, ($H_6$), and moreover $f_i(\cdot,0,\ldots , 0 )\geq 0$ for all $1\leq i\leq n$.
Let $(u_1, \ldots , u_n)$  be a viscosity solution of \eqref{P}, where $u_i$'s are  $C^1(\overline{\Omega})$ functions. 
Then $u_i$ is radially symmetric and $\partial_r u_i <0$, for all $r=|x|\in (0,R)$, $1\leq i\leq n$. 
\end{teo}

We stress that viscosity solutions of fully nonlinear equations usually have $C^{1,\alpha}$ regularity if the operator is continuous with respect to $x$ and has structure as in $(H_1)$.  In the presence of a superlinear and at most quadratic gradient growth this follows from \cite{regularidade}.  We prefer to keep hypothesis $(H_2)$, which is more general, once symmetry in itself does not require some particular growth. On the other hand, the hypothesis on the sign of $f_i$ is quite natural when considering moving planes throught fully nonlinear structures, see \cite{BQnonproper, BSoverdet}.

\smallskip

As an application, we obtain uniqueness results for the following Lane-Emden system posed in a ball,
\begin{align} \label{LE} \tag{LE}
\left\{
\begin{array}{rclcc}
\Delta u +v^q &=& 0 &\mbox{in} & \;B \\
\Delta v +u^p &=& 0 &\mbox{in} & \;B \\
u,\, v &>& 0 &\mbox{in} & \;B \\
u,\, v &=& 0 &\mbox{on} & \partial B
\end{array}
\right.
\end{align}
where $p,q>0$, $pq\neq 1$. Namely, $pq=1$ is a separate case related to eigenvalue problems, in which the scaling $(t^q u,tv)$, $t \in (0, \infty)$, produces multiple eigenfunctions; see \cite{Dalmasso, Montenegro} for instance.

It is well known that criticality plays a meaningful role in existence and nonexistence results. As far as problem \eqref{LE} is concerned, we will be interested in the subcritical case
\begin{align}\label{subcritico}
\frac{N}{p+1}+\frac{N}{q+1}>N-2,  \ \ p,q>0, \ \ N\geq 1,
\end{align}
since no solution exists in the complementary setting; see for example \cite[Proposition 3.1]{Mitidieri93}. Several papers prove existence of  positive solution in the subcritical case for a general bounded regular domain; see for instance \cite[Theorems 1.3 and 1.4]{EdersonTAMS2012} and \cite{EdersonPortugaliae} for a rather complete overview on the subject. Uniqueness is known when $pq<1$ for general bounded smooth domains; see \cite[Theorem 4.1]{Montenegro} or \cite[Theorem 7.1]{EdTran2018}. 
However, we cannot expect a uniqueness result holding true for general domains if $pq>1$, since multiplicity has already been proved for an annulus in \cite[Theorem 1.1]{AnelNonuniq}.
In the case of a ball, uniqueness of positive radial solution follows from \cite[Theorem~1.1 (i)]{Dalmasso}. Here we complement these results.

\begin{corol}\label{uniqLE}
Let $B$ be a ball and $p,q>0$. Then,
\begin{enumerate}[(i)]
\item Every pair of solutions $u,v \in C^2(B)\cap C(\overline{B})$ of \eqref{LE} is radially symmetric and strictly decreasing;

\item Under \eqref{subcritico} with $pq\neq1$, the problem \eqref{LE} has a unique solution $u,v \in C^2(B)\cap C(\overline{B})$. 
\end{enumerate}
\end{corol}

We stress that uniqueness for positive solutions in the superlinear case $pq>1$, under the extra condition $p,q\geq 1$, follows from \cite[Theorem 1]{Troy} combined with \cite[Theorem 1.1 (i)]{Dalmasso}.  The novelty here is the extension for $p,q>0$ in which one of them is strictly less than one, situation where differentiability or Lipschitz condition is no longer true.

\smallskip

The rest of the paper is organized as follows. In the preliminary Section \ref{Preliminaries} we introduce some notations and briefly recall the moving planes method. In Section  \ref{secao proofs} we develop the proofs of the main theorems. 
Section~\ref{secao aplicacoes} is devoted to further discussions and applications. It includes results for equations involving Pucci's extremal operators and a description for a continuum of solutions treated in \cite{multiplicidade}, which comes from a class of problems  with natural growth in the gradient.

\section{Preliminaries and notations}\label{Preliminaries}

Assume that $F_i:\overline{\Omega}\times\rN\times\mathbb{S}^N(\mathbb{R})\rightarrow\real$ satisfies $(H_1)$, for all $i\in \{1,\ldots, n\}$. 
In $(H_1)$,
$$
\mathcal{M}^+_i(X):=\sup_{\alpha_i I\leq A\leq \beta_i I} \mathrm{tr} (AX)\,,\quad \mathcal{M}_i^-(X):=\inf_{\alpha_i I\leq A\leq \beta_i I} \mathrm{tr} (AX)
$$
are the Pucci's extremal operators with ellipticity constants $0<\alpha_i\leq \beta_i$, $i\in \{1,\ldots,n\}$.
See, for example, \cite{CafCab} for their properties. In particular, the condition $(H_1)$ on the $X$ entry means that $F_i$ is a uniformly elliptic operator, uniformly continuous in $(p,X)$. 

Throughout the text we denote 
\begin{equation}\label{defLcomPucci}
\mathcal{L}^\pm [w]:=\mathcal{M}_1^\pm (D^2 w)\pm \gamma|Dw|, \quad \mathcal{\widetilde{L}}^\pm [w]:=\mathcal{M}_1^\pm (D^2 w)\pm (\gamma +\mu )|Dw|, 
\end{equation} 
where $\mu=\mu_{\sigma}$, with $\sigma = 2\|w\|_{C^1(\overline{\Omega})}$, comes from hypothesis $(H_2)$.
Moreover, $\partial_\nu$ is the derivative in the direction of the interior unit normal vector. 
Along the text we will use the strong maximum principle and Hopf lemma for single equations involving Pucci's operators in the form \eqref{defLcomPucci}, which we write SMP and Hopf for short; for instance see \cite{BardiF}. 

Now we recall the definition of viscosity solution, understood in the $C$-viscosity sense; see also \cite{CafCab}.
\begin{defin}\label{def Lp-viscosity sol}
We say that $\textbf{u}=(u_1,\ldots , u_n)\in (C(\Omega))^n$ is a viscosity subsolution $($supersolution$)$ of  
\begin{center}
$\mathcal{F}_i(x,\textbf{u},Du_i,D^2 u_i)=0$\, in\, $\Omega$, \quad $i=1, \ldots ,n$,
\end{center}
if, for each $i\in \{1,\cdots , n\}$, whenever $\phi\in  C^{2}(B_s(x_0))$ is such that $u_i-\phi$ have a local maximum $($minimum$)$ at $x_0$, it follows
$\mathcal{F}_i(x_0,\textbf{u}(x_0),D\phi(x_0),D^2\phi (x_0))  \geq 0\;\,
( \mathcal{F}_i(x_0,\textbf{u}(x_0),D\phi(x_0),D^2\phi (x_0))  \leq 0)$. 
\end{defin}

Next we recollect the main ingredients and precise assumptions for the moving planes method in the $x_1$ axis direction. For the time being, consider $\Omega$ as a bounded $C^{2}$ domain.
Assume that $\Omega$ is convex in the $x_1$ direction and symmetric with respect to the plane $x_1=0$, that is,
\begin{align}\label{Omega x1}
\textrm{if \,} (x_1\ \ldots , x_N) \in \Omega, \textrm{\, then\, } (t,x_2,\ldots,x_N)\in \Omega , \textrm{ for all\, } t\in [-x_1,x_1].
\end{align}

Our tools are the parallel hyperplanes $$T_\lambda=\Omega\cap \{x_1=\lambda\}.$$
Let $R=\sup_{x\in \Omega}\, x_1 $, where $x = (x_1, \ldots, x_N)$. 
For $\lambda <R$ we  consider the right cap $$\Sigma_\lambda =\Omega\cap \{x_1 >\lambda\},$$ i.e.\ the part of the semiplane located on the right hand side of $T_\lambda$ which is in $\Omega$.
Moreover, we denote by $A^\lambda$ the reflection of a set $A$ with respect to the plane $\{x_1=\lambda\}$.

We start decreasing $\lambda$ from $R$, by moving the plane $\{x_1=\lambda\}$ from right to left as far as $\{x_1=\lambda\}$ intersects $\overline{\Omega}$ with $\Sigma_\lambda^\lambda \subset \Omega$. From the usual index notation \cite{GNiN, Troy}, we denote
\begin{align*}
\Lambda_2 = \inf \{ \lambda < R\,; \; \Sigma_\mu^\mu \subset \Omega, \textrm{ for all } \mu \in (\lambda, R) \, \}.
\end{align*}

The major challenge in the moving planes technique that goes back to Gidas, Ni and Nirenberg is to deal with the reflection of the \textit{boundary} of $\Sigma_\lambda$, specifically the part of the boundary that is on $\partial\Omega$. 
Namely, $(\partial\Sigma_\lambda \cap  \{x_1>\lambda\} )^\lambda$ is contained within the domain $\Omega$ for $\lambda$ sufficiently close to $R$, due to the regularity of the domain. 

As in \cite{GNiN}, when we decrease the values of $\lambda$, it happens that $\Sigma_\lambda$ reaches a position in which at least one of the following situations occurs for the first time:
\smallskip

(I) $ \Sigma_\lambda^\lambda$ becomes internally tangent to $\partial\Omega$ at some point which is not on $\{x_1=\lambda\}$;
\smallskip

(II) $\{x_1=\lambda\}$ reaches a position where it is orthogonal to $\partial\Omega$ at some point. 

\medskip

Such a value of $\lambda$ is denoted by $\Lambda_1$ (of course $\Lambda_1 =0$ if $\Omega$ is a ball centered at $0$), that is, 
\begin{align*}
\Lambda_1 = \inf \{ \lambda < R\,; \; \Sigma_\mu  \textrm{ does not reach positions (I) and (II), for all } \mu \in (\lambda, R) \, \}.
\end{align*}
Note that $\Sigma_{\Lambda_1}^{\Lambda_1}\subset \Omega$, and the limiting position $\Lambda_2$ can be less than $\Lambda_1$. In general, $\Lambda_2\leq \Lambda_1$. Observe that, if $\Omega$ satisfies \eqref{Omega x1}, then $R>0$, $\Lambda_2=0$; and if $\Lambda_1$ was positive, then it would happen at a point in which the plane $\{x_1=\Lambda_1\}$ is orthogonal to $\partial\Omega$. 
 
The preceding difficulty never appears in the approach of Berestycki-Nirenberg for the moving planes method, since they do not need to take into account the behavior of the reflection of $\partial\Sigma_\lambda$, but only of $\Sigma_\lambda$ itself. This is possible through the maximum principle for small domains, which is applied in a neighborhood of $\partial\Omega$. Nevertheless, such an approach requires a Lipschitz condition on $f_i$ which we are not assuming.

In any case we define
$$
U_i^\lambda =u_i^\lambda-u_i \, , \;\;\textrm{ for }\;\, u_i^\lambda (x)=u_i(x^\lambda),
$$
where $x^\lambda := (2\lambda -x_1, x')$ is the reflection of the point $x$ with respect to the plane $\{x_1=\lambda\}$, for each $x=(x_1, x')\in \Omega\,$; here $x^\prime =(x_2,\ldots,x_N)\in \mathbb{R}^{N-1}$.
\smallskip

Regarding the $x_1$ direction, as in \cite{BdaLio} we denote 
$\bar{p}=(-p_1,p_2,\ldots , p_N)$, and $\bar{X}$ as the matrix with entries $\epsilon_{\iota j}X_{\iota j}$,
where $\epsilon_{11}=1$, $\epsilon_{\iota j}=1$ if $\iota,j\geq 2$, 
and $\epsilon_{1j}=\epsilon_{j1}=-1$ if $j\neq1$, for any $p=(p_1,\ldots,p_N)\in \rN$ and for any symmetric matrix $X=(X_{\iota j})_{\iota,j}\in \mathbb{S}^N(\mathbb{R})$. 
Observe that $X$ and $\bar{X}$ have the same eigenvalues. We then assume as in \cite{BN},
instead of $(H_0)$, 
\smallskip
\begin{itemize}
\item[$(\widetilde{H}_0)$] $F_i\,(y_1,x^\prime,\bar{p},\bar{X}) + f_i\,(y_1,x^\prime,\textbf{u}, \bar{p})\geq F_i \,(x,p,X)+ f_i \,(x,\textbf{u},p)$, for all $p,X$, $\textbf{u}=(u_1,\ldots, u_n)$, and for $x=(x_1,x^\prime)\in \Omega$ such that $y_1< x_1$ with $y_1+x_1>0$.
\end{itemize}
Notice that $(\widetilde{H}_0)$ itself comprises monotonicity in the $x$ entry with respect to the $x_1$ direction; so we do not need to assume a version of $(H_3)$ in the $x_1$ direction.

Notice that $x_1> x_{1}^\lambda> -x_1$ for all $x\in \Sigma_\lambda$ when $\lambda >\Lambda_2$. Thus, it follows from $(\widetilde{H}_0)$ that, for any solution $(u_1,\ldots , u_n)$ of \eqref{Plambda}, $u_i^\lambda$ satisfies the following inequality
\begin{align}\label{eq ulambda}
-F_i\,(x, Du_i^\lambda(x),D^2u_i^\lambda(x)\,) \geq f_i\,(x, u_1^\lambda(x),\ldots, u_n^\lambda(x), Du_i^\lambda(x)\,) \; \textrm{ in } \Sigma_\lambda .
\end{align}

\medskip

Further, as pointed out in \cite{GNiN}, we need the following hypothesis on $f_i(\cdot,0,\ldots,0)$ in the nonradial case. 

\smallskip
\begin{itemize}
\item[$(H_7)$] For each $i\in \{1,\ldots,n\}$, on $\partial\Omega\cap \{x_1>0\}$ we have either $f_i \,(\cdot,0,\ldots,0)\geq 0$ or $f_i\, (\cdot,0,\ldots,0)< 0$.
\end{itemize}
\smallskip
Note that such condition is trivially satisfied in the radial scenario since $f(x,0, \ldots,0)$ is constant on $\partial \Omega$. Moreover, set \eqref{Hx1} as
\begin{align}\label{Hx1}\tag{$\widetilde{H}$}
(\widetilde{H}_0),\; (H_1),\;(H_2),\;  (H_4), \;(H_5), \textrm{ and } (H_7).
\end{align}

To finish the section, we observe that assumption $(H_5)$ on $f_i$ is equivalent to
$c_{ij} \geq 0$ for all $i\neq j$, in which the function $c_{ij}$ is defined as 
\begin{equation}\label{cross}
c_{ij}(x,h,p)=\frac{1}{h} \,\{ f_i (x,\textbf{w}^j,p)-f_i (x,\textbf{u},p)\} \textrm{\, if\, } h\neq 0, \textrm{ \,and\, } c_{ij}(x,0,p)=0,
\end{equation}
where $\textbf{w}^j=\textbf{w}^j(\textbf{u},h)=(w_1,\ldots , w_n)$ with $w_j=u_j+h$, $w_l=u_l$ if $l\neq j$ and $\textbf{u} = (u_1, \ldots, u_n)$.  
We also denote the standard orthonormal basis in $\rN$ as $\{e_l\}_{ 1\leq l \leq  N }$. 
\smallskip

\begin{rmk}\label{Remark lambda>0}
Notice that, for $\lambda >\Lambda_2$, the regularity of the domain implies the existence of some points in the reflection of the boundary portion $ \partial \Sigma_\lambda \cap \partial \Omega$ which lie in $\Omega$. So, at such points we have $U_i^\lambda>0$ for all $i\in\{1,\ldots,n\}$. 
In particular, we can never have $U_i^\lambda \equiv 0$ in $\Sigma_\lambda$ for some $i$ in this case.
If, moreover, $\lambda >\Lambda_1$, then $ (\partial \Sigma_\lambda \cap \{x_1>\lambda\})^\lambda \subset \Omega$.
\end{rmk}

\section{Proofs of the main theorems}\label{secao proofs}

We start with a symmetry result that comprises a class of domains for which symmetry can be obtained in one direction without additional hypotheses on the operators $F_i$, for example if $\Omega$ is an ellipse.

\begin{prop}\label{th elipse}
Let $\Omega$ be a $C^{2}$ domain satisfying \eqref{Omega x1}, where situation $\mathrm{(II)}$ in the preceding section does not occur for any $\lambda >0$.
Assume \eqref{Hx1}, for all $1\leq i\leq n$.
Let $(u_1, \ldots , u_n)$  be a solution of \eqref{P}, where $u_i$'s are $C^2(\overline{\Omega})$ functions. 
Then $u_i$ is symmetric with respect to the plane $x_1 =0$, and $\partial_{x_1} u<0$ for any $x_1>0$, for all $1\leq i\leq n$.
\end{prop}

Notice that a ball is the simplest example of a domain satisfying the hypotheses of Proposition \ref{th elipse}.
Furthermore, observe that this proposition and $(H_0)$ produce the radial symmetry and monotonicity of Theorem \ref{Tbola}. Indeed, concerning an arbitrary direction $w$, we reduce it to the case of the $x_1$ direction via rotation $x\mapsto xQ$, where $Q$ is a $N \times N$ orthogonal matrix such that $wQ =e_1$. The symmetry of $u_i$ in the direction $w$ follows from the symmetry of $W_i(x) := u_i(xQ)$ with respect to the plane $x_1=0$. 

If, moreover, the operators $F_i$,  $1 \leq i \leq n$, are continuously differentiable in the entry $X \in \mathbb{S}^N(\mathbb{R})$, the following more general symmetry result in one direction can be accomplished.

\begin{teo} \label{Tx1}
Let $\Omega$ be a bounded $ C^{2}$ domain satisfying \eqref{Omega x1}. Assume that $F_i\,(x,p,X)$ is continuously differentiable in $X$, and that $f_i$ and $F_i$ verify \eqref{Hx1}, for all $1\leq i\leq n$.
Let $(u_1, \ldots , u_n)$  be a solution of \eqref{P}, where $u_i$'s are  $C^2(\overline{\Omega})$ functions. Then, for each $1\leq i\leq n$, $u_i$ is symmetric in the $x_1$ direction and $\partial_{x_1} u_i <0$, for all $x\in \Omega$ with $x_1>0$. 
\end{teo}

For ease of notation, we provide a proof in the case $n=2$, with $u_1,u_2$, $f_1,f_2$, and $F_1,F_2$ are replaced by $u,v$, $f,g$, and $F,G$, respectively. 
We indicate the respective changes that appear from the incorporation of other components when necessary.
The system is rewritten as
\begin{align}\label{Q} \tag{Q}
\left\{
\begin{array}{rclcc}
F (x,Du,D^2u) +f (x,u,v, Du)&=&0 &\mbox{in} & \;\Omega \\
G (x,Du,D^2u) +g (x,u,v, Dv)&=&0 &\mbox{in} & \;\Omega \\
u,\, v\in C^2(\overline{\Omega}), \,\;\; u,\, v &>& 0 &\mbox{in} & \;\Omega \\
u,\, v &=& 0 &\mbox{on} & \partial\Omega.
\end{array}
\right.
\end{align}
For the decomposition in $(H_4)$ we just rewrite $f=f_1+f_2$ and $g=g_1+g_2$.

\smallskip

Then we consider
$
U^\lambda :=u^\lambda-u , \; V^\lambda :=v^\lambda - v,
$
for $u^\lambda (x)=u(x^\lambda)$, $v^\lambda (x)=v(x^\lambda)$.

We will need the following lemma, which is a consequence of SMP and Hopf. It concerns the domain $\Sigma_\lambda$ when $\lambda >\Lambda_2$, and solutions of \eqref{Q}. 

\begin{lem}{$($SMP and Hopf in $\Sigma_\lambda )$}\label{SMP e Hopf sigma lambda}
Let $\lambda >\Lambda_2$ and let $u,v$ be a pair of solutions of \eqref{Q}. Assume $U^\lambda$ and $V^\lambda$ nonnegative in $\Sigma_\lambda$ and  \eqref{Hx1}.
\smallskip
Then $U^\lambda, V^\lambda> 0$ in $\Sigma_\lambda$ and $u_{x_1} ,\, v_{x_1} <0$ on $T_\lambda$. 

\end{lem}

\begin{proof}
Using $(H_1)$, \eqref{defLcomPucci}, \eqref{eq ulambda}, and $(H_2)$ we obtain
\begin{align}\label{conta f}
-\mathcal{{L}}^- [\,U^\lambda\,] & \geq - \{\, F(x,Du^\lambda, D^2 u^\lambda\,)-F(x,Du ,D^2 u) \, \} \geq f(x ,u^\lambda,v^\lambda, Du^\lambda)-f(x, u, v, Du)\nonumber \\
& = \{ f(x,u^\lambda,v^\lambda, Du^\lambda)-f(x,u^\lambda,v, Du^\lambda) \}
+\{f_1(x,u^\lambda,v,Du^\lambda) - f_1(x, u, v,Du^\lambda) \}
 \nonumber\\
&  +\{f_2(x,u^\lambda,v,Du^\lambda) - f_2(x, u, v,Du^\lambda) \}
+\{f(x ,u,v, Du^\lambda) -f(x, u, v, Du)\}  \nonumber \\
& \geq c_{12} (x,V^\lambda,Du^\lambda) \,V^\lambda  
-d_f \, U^\lambda -\mu |DU^\lambda|  \textrm{ in } \Sigma_\lambda,
\end{align}
since $f_2$ is nondecreasing in $u$. Here $d_f$ is the uniform Lipschitz constant of $t\mapsto f_1(x,t, v,p)$ for $t \in [0,\sup_{\Omega} u]$, $p\in \overline{B}_{\sigma} (0)$ with $\sigma =2\|u\|_{C^1(\overline{\Omega})}$ and $c_{12}$ comes from \eqref{cross}. 
Using further $c_{12}(x,V^\lambda,Du^\lambda)V^\lambda \geq 0$, we derive
$\mathcal{\widetilde{L}}^- [\,U^\lambda\,] \leq d_f \, U^\lambda$, see \eqref{defLcomPucci}. 
Then, SMP entails either $U^\lambda >0$ in $\Sigma_\Lambda$ or $U^\lambda \equiv 0$ in $\Sigma_\lambda$.
Thus, by Remark~\ref{Remark lambda>0}, $U^\lambda >0$ in $\Sigma_\lambda$.
Since the interior unit normal vector on $\partial \Sigma_\lambda \cap T_\lambda$ is $\nu = e_1$, Hopf yields $0< \partial_{\nu} \, U^\lambda = -2\, u_{x_1}$ on $T_\lambda$. 

The proof of $V^\lambda >0$ in $\Sigma_\lambda$ is analogous, by considering \eqref{Hx1} for $g$, and $d_g$ as the uniform Lipschitz constant of $s \mapsto g_1(x,u,s,p)$ for $s \in [0,\sup_{\Omega} v]$, $p\in \overline{B}_{\sigma} (0)$.
\end{proof}

\begin{proof}[Proof of Proposition {\rm{\ref{th elipse}}}.] 
We split the proof in two steps.

\smallskip

{ {\textit{Step 1: Start moving the planes.}}}
\smallskip

Notice that the first component of the interior unit normal vector at a boundary point $x$, denoted by $\nu_1(x)$, is negative for any $x\in \partial\Omega$ which is close to points of the form $(R,x^\prime)$, where $R=\sup_{x\in\Omega} x_1$ from Section \ref{Preliminaries}. Fix $x_0=(R,x_0^\prime)$ and $\varepsilon>0$ such that this property remains true for all points on $\partial \Omega \cap B_\varepsilon (x_0)$.

By taking a smaller $\varepsilon$ if necessary (independent of $u$), we are going to show that 
\begin{equation}\label{eq:AF1}
u_{x_1}<0 \;\textrm{ on }\; \Omega \cap B_\varepsilon (x_0).
\end{equation}
Repeating the same for $v$, it will imply $U^\mu,V^\mu >0$ in $\Sigma_\mu$, $\mu\in (\lambda,R)$, for values of $\lambda$ close to $R$.

\smallskip

\textit{Case 1. Assume first $f(x,0,0,0)\geq 0$, for all $x\in \Omega$.}
In this case, since $F(x,0,0)\equiv 0$, 
\begin{align*}
-\mathcal{L}^-[u]  &\geq - F(x,Du,D^2 u)=f(x,u,v,Du)\geq f(x,u,v,Du)-f(x,0,0,0)\\
& =\{ f(x,u,v,Du)-f(x,u,0,Du)\}+\{f_1(x,u,0,Du)-f_1(x,0,0,Du)\} \\
& +\{f_2(x,u,0,Du)-f_2(x,0,0,Du)\}+\{f(x,0,0,Du)-f(x,0,0,0)\}\\
&\geq c_{12}(x,v,Du)v-d_f\, u - \mu |Du|.
\end{align*}
Then $\mathcal{\widetilde{L}}^- [u] \leq d_f  u$ and Hopf give rise to $\partial_\nu u >0$ on $\partial\Omega$; by a covering argument we obtain \eqref{eq:AF1}. 

\smallskip
\textit{Case 2. Suppose that there exists some $\tilde{x}\in \Omega$ such that $f(\tilde{x},0,0,0)<0$. }

In this case, by assuming that \eqref{eq:AF1} is not true, there exists a sequence of points $z_k$ in $\Omega$ converging to $x_0$ such that $ u_{x_1}(z_k)\geq 0$. 
Observe that $u>0$ in $\Omega$ and $u = 0$ on $\partial \Omega$ imply $ u_{x_1} (x_0) \leq 0$. Hence, $u_{x_1}(x_0)=0$. This and the fact that the gradient of $u$ is parallel to the unit normal vector  at $x_0$ yield $Du(x_0)=0$.

For each $k$, the segment in the positive $x_1$ direction from $z_k$ intersects $\partial\Omega \cap B_\varepsilon (x_0)$ at a point $y_k$ where $ u_{x_1} (y_k)\leq 0$. The mean value theorem gives $\xi_k$ such that
$
0\leq u_{x_1}(z_k)-u_{x_1}(y_k)=u_{x_1 x_1}(\xi_k) \,(z_{k,1}-y_{k,1}),
$
from where $ u_{x_1 x_1}(\xi_k)\leq 0$, and so $u_{x_1 x_1}(x_0)\leq 0$ in the limit.

\smallskip
Fix an arbitrary $t_0\in T_{x_0}$.
As in \cite{Troy}, let $T_{x}$ be the tangent space to $\partial\Omega$ at $x$, and let $x(s)$ be a fixed path on $\partial\Omega$ such that $x(0)=x_0$, $\dot x (0)=t_0$. From $u(x(s))=0$ for all $s$, it follows that $Du(x(s))\cdot \dot x (s)=0$ for all $s$ as well. Furthermore,
\[
Du(x) \cdot t =0 \textrm{ for all } x \in \partial \Omega \textrm{ and } t \in T_x.
\]
Take some $t_1\in T_{x_0}$ and let $t(s)$ be a path with  $t(s) \in T_{x(s)}$ and $t(0)=t_1$. From $Du(x(s))\cdot t(s)=0$ for all $s$, by differentiating it and using $Du(x_0)=0$, yields
\begin{align}\label{4.8}
t_1\cdot (D^2u(x_0)\cdot t_0)=0.
\end{align}
Next, since the function $\varphi (s):=Du(x(s))\cdot \nu (x(s))=\partial_\nu u(x(s))$ has a minimum equal to zero at $0$, then $\varphi^\prime (0)=0$, from where
$\nu (x_0)\cdot (D^2u(x_0)\cdot t_0)=0$.
This and \eqref{4.8} imply that $D^2u(x_0)\cdot t_0=0$, where $t_0\in T_{x_0}$ is arbitrary. Consequently, $0$ is an eigenvalue of order $N-1$ for $D^2 u (x_0)$. 
Namely, let $\{e,0,\ldots,0\}$ be the spectrum of $D^2 u (x_0)$.

We claim that $e>0$. Indeed, if $e\leq 0$, then $\mathcal{M}^+_1(D^2u(x_0))=e\alpha_1 \leq 0$. Hence
\begin{align}\label{conta f com x}
0\leq -\mathcal{M}^+_1(D^2u(x_0)) \leq -F(x_0,0,D^2 u(x_0)\,)=f(x_0,0,0,0).
\end{align}
Consider the point $\tilde{z}$ in which the line segment on the $x_1$ direction from $\tilde{x}$ hits $\partial\Omega$. By monotonicity,
\begin{align}\label{z tilde negativo}
f(\tilde{z},0,0,0)\leq f(\tilde{x},0,0,0)<0.
\end{align}
Clearly, \eqref{conta f com x} and \eqref{z tilde negativo} contradict the hypothesis $(H_5)$. Thus, the claim $e>0$ is proved. 

Now, a basis of $\real^N$ composed of $\nu(x_0)$ and $\{a_l\}_{1\leq l\leq N-1}$ (orthonormal basis of $T_{x_0}$) applied to the above calculations ensures that $D^2 u(x_0)=(e \,\nu_i \nu_j)_{ij}$. 
In particular, $u_{x_1 x_1} (x_0)=e\,\nu_1(x_0)\nu_1(x_0)$ is positive, which contradicts $u_{x_1x_1}(x_0)\le 0$. We so conclude \eqref{eq:AF1} and Step 1.

\medskip

Actually, Step 1 is completely independent and it only uses the regularity of the solutions up to the boundary, along with the regularity of the boundary itself. That is why we kept it in terms of $\Lambda_2$. Recall $\Lambda_2=\Lambda_1=0$ under hypotheses of Proposition \ref{th elipse}. 

Therefore, by Step 1, it is well defined the following quantity
\begin{equation*}
\Lambda = \inf \{ \,\lambda >0; \; U^\mu >0, \; V^\mu >0 \textrm{ \,in\, } \Sigma_\mu\, ,  \textrm{ for all } \mu\in (\lambda,R)\,\}.
\end{equation*}
By continuity, $U^\Lambda \geq 0$ and $V^\Lambda \geq 0$ in $\Sigma_\Lambda$. 
\medskip

{ {\textit{ Step 2: Stop moving the planes at zero.}}}
\smallskip

The goal is to show that $\Lambda =0$. Suppose $\Lambda >0$ in order to obtain a contradiction. Thus, by Lemma \ref{SMP e Hopf sigma lambda}, we have that $U^\Lambda, V^\Lambda>0$ in $\Sigma_\Lambda$ and $u_{x_1} , v_{x_1} <0$ on $T_\Lambda$. Furthermore, there exists some $\varepsilon>0$ such that
\begin{equation}\label{eq:newderivative}
u_{x_1} , v_{x_1} <0 \;\textrm{ on }  \;T_{\Lambda-\varepsilon}.
\end{equation}
Indeed, $u_{x_1} , v_{x_1} <0$ on $T_{\Lambda-\varepsilon}\cap K$, for any compact $K\subset \Omega$, with $\varepsilon=\varepsilon (K)$. On a neighborhood of the boundary points $x\in \partial\Omega\cap \{x_1=\Lambda-\varepsilon\}$, we can apply exactly the same argument used to derive \eqref{eq:AF1}. This ensures \eqref{eq:newderivative}. 

Consider, then, a sequence of points $x_k\in \Sigma_{\lambda_k}$, with $\lambda_k\in (\Lambda - \varepsilon , \Lambda)$, for some $\varepsilon \in (0,\Lambda)$, $\lambda_k \rightarrow \Lambda$ as $k\rightarrow \infty$, such that, for each $k\in \n$, either $U^{\lambda_k} (x_k)\leq 0$ or $V^{\lambda_k}(x_k)\leq 0$.

\smallskip
By passing to a subsequence, say $x_k \rightarrow z \in \overline{\Sigma}_\Lambda$ and $U^{\lambda_k}(x_k)\leq 0$ (at least one of them verifies this for infinite $k$'s). 
Thus, $U^\Lambda (z)\leq 0$.
Now, since $U^\Lambda >0$ in $\Sigma_\Lambda \cup (\partial\Sigma_\Lambda\cap\{x_1>\Lambda\})$, we must have $z\in \partial\Sigma_{\Lambda}\cap \{x_1=\Lambda\}$ and, from \eqref{eq:newderivative}, $u_{x_1}(z)\leq 0$.
Further, since the line segment between $x_k$ and $x^{\lambda_k}_k$ is contained in $\Omega$, the mean value theorem yields $y_k \in \Omega$ such that
\begin{align*}
0\geq U^{\lambda_k} (x_k)= u(x_k^{\lambda_k})-u(x_k)= 2\, u_{x_1}(y_k) (\lambda_k -x_{k,1}).
\end{align*}
Hence $u_{x_1}(y_k)\geq 0$. Here $y_k \rightarrow z$, since both $x_k$ and $x^{\lambda_k}_k$ converge to $z=z^\Lambda$, and this contradicts \eqref{eq:newderivative}.
\end{proof}

For the proof of Theorem \ref{Tx1} we recall the following lemma which is needed in the usual treatment of Serrin's corner lemma for nonlinear operators, as in \cite{GNiN, Serrin, BSoverdet}.

\begin{lem}\label{uso serrin}
Assume that $F_i(x,p,X)$ is $C^1$ in $X$ and satisfies $(\widetilde{H}_0)$, $(H_1)$, for some $i$. Let $\lambda$ be such that $U_i^\lambda>0$ in $\Sigma_\lambda$. Then $\partial^2_{s}\, U_i^\lambda (z)>0$ for any direction $s$ that enters in $\Sigma_\lambda$ at $z$ nontangentially, at any point $z\in \partial\Omega\cap\{x_1=\lambda\}$ in which $T_\lambda$ is orthogonal to $\partial\Omega$.
\end{lem}
\begin{proof} Set $p=p(x)=Du_i(x)$. Then by splitting 
$$ \{ F_i\,(x,D u_i^\lambda, D^2 u^\lambda_i) - F_i\,(x,p, D^2 u^\lambda_i) \}+\{F_i\,(x,p, D^2 u^\lambda_i) -F_i\,(x,p, D^2 u_i) \},
$$
and using calculations from equation \eqref{conta f},  $U_i^\lambda$ becomes a solution of a linear equation in the form
\begin{align*}
-\mathrm{tr} (\,A(x)D^2 U_i^\lambda\,)+b(x)\cdot DU_i^\lambda\geq -d_f\, U_i^\lambda \;\textrm{ in\, } \Sigma_\lambda
\end{align*}
where $b$ is bounded, and $A=(a_{ jk})_{ jk}$ has continuous entries in $\overline{\Sigma}_\lambda$. Precisely, $b(x)= DU_i^\lambda(x)$ $(\mu+{\gamma}/ |DU_i^\lambda(x)|) $ if $|DU_i^\lambda(x)|\neq 0$, and $b(x)=0$ otherwise; 
\begin{align*}
&a_{ jk}(x)=\int_0^1 \partial_{X_{jk}} F_i(x,p,tD^2u_i(x)+(1-t)D^2 u_i^\lambda (x) ) \; \mathrm{d} t\\
&=\frac{1}{2}\int_0^1 \{ \partial_{X_{jk}} F_i(x,p,tD^2u_i(x)+(1-t)D^2 u_i^\lambda (x) ) +
 \partial_{X_{jk}} F_i(x,p,tD^2u_i^\lambda (x)+(1-t)D^2 u_i (x) )  \}\; \mathrm{d} t.
\end{align*}

Since $F_i(x,p,X)=F_i(x,p,\bar{X})$, it follows that $\partial_{X_{1 k}}F_i (x,p,X)+\partial_{X_{1 k}}F_i (x,p,\bar{X})=0$ for $k>1$. Hence, 
$a_{1 k}=0$ on $T_\lambda\cap \partial \Sigma_\lambda$ for all $k>1$.
An application of Serrin's lemma \cite[Lemma 2]{Serrin}, at any point $z\in \partial\Omega\cap\{x_1=\lambda\}$ in which $T_\lambda$ is orthogonal to $\partial\Omega$, yields
\begin{align*}
\partial_{s} U_i^\lambda (z)>0 \;\textrm{ or } \; \partial^2_{s} \,U_i^\lambda (z)>0,
\end{align*}
for any direction $s$ that enters in $\Sigma_\lambda$ at $z$ nontangentially.
At such $z$, however, the functions $u_i^\lambda$ and $u_i$ have the same normal derivative to $\partial\Omega$, and zero tangential derivatives, thus $\partial_s U_i^\lambda (z)=0$. Therefore, $\partial^2_{s}\, U_i^\lambda (z)>0$. 
\end{proof} 

\begin{proof}[Proof of Theorem {\rm \ref{Tx1}}.]
The start of moving planes is identical to \textit{Step 1} in the previous proof, since for $\lambda$ sufficiently close to $R$ the domain does not reach position $\Lambda_1$.
For \textit{Step 2}, we need to take situation (II) into account. Assume $\Lambda >0$ in manner to achieve a contradiction.
By Lemma \ref{SMP e Hopf sigma lambda} we have $U^\Lambda>0$ in $\Sigma_\Lambda$ and $u_{x_1} <0$ on $T_\Lambda$.
Analogously to the previous \textit{Step 2}, since $\Lambda_2=0$ is not the infimum $\Lambda$, there are sequences $\lambda_k \in (0,\Lambda)$, $\lambda_k \rightarrow \Lambda$, and $x_k \in \Sigma_{\lambda_k}$ such that $u(x_k)\geq u(x_k^{\lambda_k})$. 
We may suppose $x_k\rightarrow z$ in $\overline{\Sigma}_\Lambda$. 
Then, $U^\Lambda (z)\leq 0$, and so $z\in \partial\Sigma_\Lambda$. 

\smallskip

\textit{Case 1:} $z\in \{x_1=\Lambda\}$. 

Notice that the line segment between $x_k$ and $x_k^{\lambda_k}$ lies in $\Omega$. Then, by the mean value theorem, there exists $y_k$ on it such that 
$0\leq u(x_k)-u(x_k^{\lambda_k})=2 u_{x_1}(y_k)(x_{k,1}-\lambda_k)$, from where $u_{x_1}(y_k)\geq 0$. 
In the limit, $u_{x_1}(z)\geq 0$, and thus $z \in \partial\Omega$. Using again that $\partial_\nu u\geq 0$ on $\partial\Omega$, we have $u_{x_1}(z)\leq 0$, thus $u_{x_1}(z)=0$.

We claim that $T_{\Lambda}$ is orthogonal to $\partial\Omega$ at $z$. 
In fact, otherwise we would necessarily  have $\nu_1 (z)<0$, which in particular is enough to apply Hopf's lemma to $u$ (at $z \in \partial \Omega$) in order to conclude 
$u_{x_1} (z)<0$; but this contradicts $u_{x_1}(z)= 0$. 

Now we are in position of applying Serrin's lemma. By Lemma \ref{uso serrin}, $\partial^2_{s}\, U^\Lambda (z)>0$, for any direction $s$ that enters in $\Sigma_\Lambda$ at $z$ nontangentially. Next we compute it for the direction $s=(1/\sqrt{2},0,\ldots,0,-1/\sqrt{2})$ properly chosen. Without loss of generality, we consider the interior normal as $\nu (z)= -e_N$. Then, at $z$,
\begin{align}\label{conta226}
\partial^2_{s} \,U^\Lambda>0 \;\Leftrightarrow\; (\partial_{x_1}-\partial_{x_N})^2 \,u^\Lambda >(\partial_{x_1}-\partial_{x_N})^2\, u.
\end{align}
Using that $u^\Lambda_{x_1 x_1}=u_{x_1 x_1}$, $u^\Lambda_{x_N x_N}=u_{x_N x_N}$, and $u^\Lambda_{x_1 x_N}=u_{x_1 x_N}$ at $z$, we obtain from \eqref{conta226} that $u_{x_1x_N} >0$ in a neighborhood of $z$. We consider the segment $I_k$ from $x_k$ to the point $z_k$ where it hits $\partial\Omega$ in the $e_N$ direction. Integration on $I_k$ for large $k$ gives $u_{x_1}(z_k)> 0$. However, from $\partial_\nu u\geq 0$ on $\partial\Omega$ we have $ u_{x_1}(z_k)\le 0$ for large $k$, which yields a contradiction.

\smallskip
\textit{Case 2:} $z\in \{x_1>\Lambda\}$.

The first thing to note is that $z^\Lambda\in \partial\Omega$, since $U^\Lambda (z)$ would be positive otherwise. Hence, since $\Lambda>0$, the interior normals $\nu(z)$ and $\nu(z^\Lambda)$ coincide, and these are orthogonal to $e_1$. As in Case~1, w.l.g.\ say $\nu(z)=\nu(z^\Lambda)=-e_N$.
Moreover, since $U^\Lambda(z)=0$, we can apply Hopf to the function $U^\Lambda$ in the domain $\Sigma_\Lambda$, from where $ U^\Lambda_{x_N} (z)<0$.

Observe that, under our contradiction assumption $\Lambda>0$ on the symmetric convex domain $\Omega$, we have $\Sigma^\Lambda_\Lambda\subset \Omega$, and so the segment $I_k$ in the $e_N$ direction from $x_k$ to $\partial\Omega$ is not longer than $I_k^{\lambda_k}$, i.e.\ $|I_k|\leq |I^{\lambda_k}_k|$. 
By the mean value theorem applied to $u$, there exists $\xi_k\in I_k$, $\tilde{\xi}_k\in I_k^{\lambda_k}$ such that $u_{x_N}(\xi_k)|I_k|=-u(x_k)\leq -u(x_k^{\lambda_k})=u_{x_N}(\tilde{\xi}_k)|I_k^{\lambda_k}|$, from where $u_{x_N}(x_k)\leq u_{x_N}(x_k^{\lambda_k})$. 
Passing to limits we obtain $u_{x_N}(z)\leq u_{x_N}(z^\Lambda)$, which contradicts  $ U^\Lambda_{x_N} (z)<0$.
\end{proof}

Next, Theorem \ref{Toverd} is a consequence of the proof of Theorem \ref{Tbola} and \cite[Theorem 2]{Troy}. The proof follows the main ideas there, which we include for the sake of completeness.

\begin{proof}[Proof of Theorem {\rm \ref{Toverd}}.]
Our aim is to prove that $\Omega$ is symmetric with respect to the plane $x_1=\lambda_1$. Then the same argument applied to any other direction will imply that $\Omega$ must be a ball, and so the symmetry is obtained via Theorem \ref{Tbola}.
Observe that exactly the same proof of Proposition \ref{th elipse}, by replacing $0$ by $\Lambda_1$ without the symmetry assumption \eqref{Omega x1} on $\Omega$, gives $\Lambda= \Lambda_1$, where
\begin{equation*}
\Lambda = \inf \{ \,\lambda >0; \; U_i^\mu >0 \textrm{ \,in\, } \Sigma_\mu\, ,  \textrm{ for all } \mu\in (\lambda,R), \, i=1,\ldots,n \,\},
\end{equation*}
and for $\Lambda_1$ as defined in Section \ref{Preliminaries}. This is due to the lack of the situation (II) for all $\lambda\in (\Lambda_1,R)$.

Next, by continuity, we have $U_i^{\Lambda_1}\geq 0$ in $\Sigma_{\Lambda_1}$ for all $i\in \{1,\ldots, n\}$. By SMP, either $U_i^{\Lambda_1}>0$ or $U_i^{\Lambda_1}\equiv 0$ in $\Sigma_{\Lambda_1}$, for each $i$. If the latter occurs for all $i\in \{1,\ldots, n\}$, then $\Omega$ is symmetric with respect to $x_1=\Lambda_1$, and the proof is finished. 

In order to obtain a contradiction, assume that $U_i^{\Lambda_1}>0$ in $\Sigma_{\Lambda_1}$ for some $i\in \{1,\ldots, n\}$.  We claim that, in this case, $\Lambda_1$ is accomplished through situation (II), that is, at a point in which the plane $\{x_1=\Lambda_1\}$ reaches an orthogonal position to $\partial\Omega$. Indeed, if situation (I) happened, then we would have $x_0^{\Lambda_1}\in \partial\Omega$, for some $x_0\in \partial\Sigma_{\Lambda_1}\cap\{x_1>\Lambda_1\}$. So, $U_i^{\Lambda_1}(x_0)=0$ and, by Hopf applied to $U_i^{\Lambda_1}$, would yield $\partial_\nu U_i^{\Lambda_1} (x_0)>0$. However, this contradicts the assumption on the overdetermined problem, namely $\partial_\nu U_i^{\Lambda_1} (x_0)=\partial_\nu u_i (x_0) - \partial_\nu u_i^{\Lambda_1}(x_0)=0$, and the claim is proved. Now Serrin's argument \cite[(a) in p. 307]{Serrin} gives us that $U^{\Lambda_1}_i$ has a zero of second order at $x_0$. Finally, a contradiction is established with Serrin's lemma as in Lemma~\ref{uso serrin} where $\lambda=\Lambda_1$.
\end{proof}

Notice that if the domain satisfies \eqref{Omega x1} and it does not have a point for which case $\mathrm{(II)}$ from Section~\ref{Preliminaries} occurs, then the proofs of Theorems \ref{Tx1} and \ref{Toverd} carry out without using Serrin's lemma.
In particular, they hold true without any differentiability hypothesis on $F_i$. In other words, this comprises a reasonable class of domains in which differentiability can be dropped. For such domains it is also possible to obtain a restatement of Theorem \ref{Tvisc} for viscosity solutions which are continuously differentiable up to the boundary. This is the content of the next proof.

\begin{proof}[Proof of Theorem \ref{Tvisc}]
As in Theorem \ref{Tbola}, we only need to prove a version of Proposition \ref{th elipse} with respect to viscosity $C^1(\overline{\Omega})$ solutions.
We first claim that $u_i$ is a viscosity solution of $F_i(x,Du_i,D^2u_i)+\bar{f}_i(x)=0$ in $\Omega$, where $\bar{f}_i (x):=f_i(x,u_1(x),\ldots,u_n(x),Du_i(x))\in C(\overline{\Omega})$. 

Let us prove the subsolution case; for the supersolution it is analogous.
Assuming the contrary, there exist some $x_0\in\Omega$ and $\phi\in C^{2}(B_s(x_0))$ such that $u_i-\phi$ has a local maximum at $x_0$ but
$F_i(x_0,D \phi(x_0),D^2 \phi(x_0))-\bar{f}_i(x_0)<0$.
By the definition of $u_i$ being a viscosity subsolution of the equation in \eqref{P}, we have that
$F_i(x_0,D \phi(x_0),D^2 \phi(x_0)) +f_i(x_0,u_1(x_0),\ldots,u_n(x_0),D\phi(x_0)) \geq 0$.
Since $u_i-\phi \in C^1(B_s(x_0))$ has a local maximum at $x_0$, we have $D (u_i-\phi)(x_0)=0$.
Thus the last two inequalities produce a contradiction.

Similarly we prove that $w_i=u_i^\lambda$ is viscosity supersolution of $F_i(x,w_1,\ldots , w_n,Dw_i,D^2w_i)+\bar{f}_i^{\lambda}(x)=0$ in $\Omega$ with $\bar{f}^\lambda_i (x):=f_i(x,u_1^\lambda(x),\ldots,u_n^\lambda(x),Du_i^\lambda(x))$, for $i=1, \ldots, n$.

Consider for simplicity $n=2$, $u_1 =u$ and $u_2 =v$. For the r.h.s. $\bar{f}_i^\lambda - \bar{f}_i$ we proceed as in \eqref{conta f}, which produces an extra term of first order $\mu |DU^\lambda|$. Meanwhile, the l.h.s. is treated in the viscosity sense by using Proposition 2.1 in \cite{BdaLio}.
In other words, $w=U^\lambda$ is a viscosity solution of $-\mathcal{L}^-[w]\geq -d_f U_\lambda -\mu |DU^\lambda|$ in $\Sigma_\lambda$.

Next we infer that such term in the r.h.s. can be retreated as being part of the Pucci operator in the l.h.s., that is, $U^\lambda$ is a viscosity solution of 
$\widetilde{\mathcal{{L}}}^- [\,U^\lambda\,] \leq d_f U^\lambda$.
Indeed, this is again a consequence of the differentiability of $u$ and the definition of viscosity solution.

Hence we carry on the rest of the proof  of Proposition \ref{th elipse} in the same way by using SMP and Hopf for viscosity solutions in order to establish the desired symmetry.
\end{proof}

\section{Discussion and further applications}\label{secao aplicacoes}

We stress that both works \cite{Troy} and \cite{Shaker} use the differentiability of $f_i$ in order to obtain radial symmetry for systems.
In fact, a typical example already discussed in Section~\ref{Introduction} is the Lane-Emden system \eqref{LE}, in which it is clearly important to contemplate cases where exponents can be less than one.

Alternatively, there are other relevant applications to systems with nondifferentiable terms. For instance, in \cite{Conti1, Conti2} the authors developed an analysis about the behavior between different species $u$ and $v$ that cohabit in $B$, in particular from the following systems
\begin{align}\label{eq.terracini}
\left\{
\begin{array}{rclcc}
-\Delta u &=& u^3 -\beta\, uv^2 \;&\mbox{in} & \;B \\
-\Delta v  &=& v^3 -\beta \, u^2 v\; &\mbox{in} & \;B \\
u,\, v &=& 0\; &\mbox{on} & \partial B.
\end{array}
\right.
\end{align}
where $N=2,3$ and $\beta\in \real$. The system  \eqref{eq.terracini} can be treated variationally as long as the right hand sides are written as $F_u (u,v)$ and $F_v (u,v)$ respectively, where
$
4F(u,v)={u^4}+{v^4} -2 {\beta} u^2v^2.
$
From their study it can be derived the more general problem
\begin{align}\label{eq.terracini.geral}
\left\{
\begin{array}{rclcc}
-\Delta u &=& u^{p-1} -\beta\, u^{r-1}v^s \;&\mbox{in} & \;B \\
-\Delta v  &=& v^{p-1} -\beta \, u^r v^{s-1}\; &\mbox{in} & \;B \\
u,\, v &=& 0\; &\mbox{on} & \partial B.
\end{array}
\right.
\end{align}
where $r+s=p$, for $2<p<2^*$, $r,s>1$, $N \geq 1$. Since the involved functions are not necessarily differentiable, our results provide (new) radial symmetry for positive solutions of \eqref{eq.terracini.geral} in the harmonious case, that is, with $\beta<0$.
On the other side, in a competitive scenario, with $\beta >0$, cooperativity is lost and our results do not apply. More than that, symmetry breaking occurs, see \cite[Remark 5.4]{TWquebra}, due to the particular segregation phenomenon in the limit $\beta\rightarrow +\infty$ described in \cite{Conti1, Conti2}.

\medskip

Now we consider the following Lane-Emden type system involving Pucci's operators
\begin{align} \label{LE pucci} 
\left\{
\begin{array}{rclcc}
\mathcal{M}_1\, u +v^q &=& 0, &\mbox{in} & \;B \\
\mathcal{M}_2 \, v +u^p &=& 0, &\mbox{in} & \;B \\
u,\, v &>& 0 &\mbox{in} & \;B \\
u,\, v &=& 0 &\mbox{on} & \partial B
\end{array}
\right.
\end{align}
where $p,q>0$, $pq\neq 1$, $\mathcal{M}_i$ can be either $\mathcal{M}^+_{\alpha_i,\beta_i}$ or $\mathcal{M}^-_{\alpha_i,\beta_i}$ in a ball.

\smallskip
In a lot of cases, a Pucci's extremal operator can feature the essence of the Laplacian.
For example, for a single equation (that is, \eqref{LE pucci} with $u=v$, $\mathcal{M}_1=\mathcal{M}_2$, $p=q$), existence of a unique classical solution was extended to Pucci's operator in \cite{Quaas2004} with locally Lipschitz nonlinearities, while existence in the sublinear case was first established in \cite{BQnonproper}. 

However, criticality relations are now given in terms of the ellipticity coefficients. Those are related to critical exponents in Liouville type results, but not completely understood even in the scalar case; see \cite{CutriLeoni, FQ2004, BQnonproper}. 
Moreover, they can be much more complicated in the case of a system.
As far as existence is concerned, define $\rho_i=(\alpha_i/\beta_i)\pm 1$ and $N_i=\rho_i (N-1)+1$, $i=1,2$.
For instance, in \cite{BQind} it was proved that there exists a positive classical solution of \eqref{LE pucci} if $p,q\geq 1$, $pq> 1$, and 
$2(p+1)/(pq -1) \geq N_1 -2$ or $2(q+1)/(pq-1) \geq  N_2 -2$, in a smooth bounded domain.

On the other hand, it is known that, in many cases, a radialization of the problem can greatly simplify the operators, specially if we are dealing with Pucci's operators. When radial assumptions on the domain and on the solutions are imposed, sometimes it is possible to go much further; see for instance \cite{FQ2004, GLPradial} and references therein. 

We stress that, regardless whether or not solutions exist, in this work we are concerned just with their radial symmetry -- in the sense that, if a solution exists then it is radial. 
In this direction, our Theorem \ref{Tbola} provides radial symmetry to solutions of \eqref{LE pucci}. In addition, it says that we can focus on establishing the unknown properties for positive solutions that are radial in nature.

\medskip

Finally, we discuss some properties of solutions of the following family of problems
\begin{align}\label{Plambda} \tag{$P_\lambda$}
\left\{
\begin{array}{rclcc}
-F(D^2u) &=&\lambda c(r)u+\mu (r) |D u|^2 +h(r) &\mbox{in} & B  \\
u&=& 0 &\mbox{on} & \partial B,
\end{array}
\right.
\end{align}
where $r = |x|$, $h\in L^\infty (\Omega)$, $F$ satisfies $(H_0)$, $(H_1)$, with $c\gneqq 0$, $\mu(r)\in [\mu_1,\mu_2]$, for some $\mu_1,\mu_2>0$. 

This class of problems is important in applications, more recently in control theory and mean field games. Moreover, it is of theoretical independent interest since this type of gradient dependence is invariant                                                                                                                                                                                           under diffeomorphic changes of variable and function, namely $x$ and $u$. 
In particular, symmetry properties might play an important role in the qualitative analysis of the set of solutions.

Now, assume that $c(r), \mu(r), h(r)$ are continuous functions in $r$ in order to treat the problem in the $C$-viscosity sense from Definition \ref{def Lp-viscosity sol}. Observe that viscosity solutions of \eqref{Plambda} belong to $C^{1,\alpha}(\overline{\Omega})$ due to \cite[Theorem 1.1]{regularidade}. Further, we have from \cite{multiplicidade} the following multiplicity result. Recall that strong solutions belong to some Sobolev space $W^{2,p}$ and satisfy the equation a.e. Here $p>N$ and so $W^{2,p}(\Omega)\subset C^1(\overline{\Omega})$.

\begin{prop} \label{th1.4}
Under the preceding hypotheses, assume that the problem $(P_0)$ has a strong solution $u_0 \geq 0$, with $c u_0\gneqq 0$.
Then every nonnegative viscosity solution of \eqref{Plambda} with $\lambda>0$ satisfies $u> u_0$. Moreover, there exists  $\bar{\lambda}\in (0,+\infty)$ such that
for every $\lambda\in (0,\bar{\lambda})$, the problem \eqref{Plambda} has at least two nontrivial strong solutions with $ u_{\lambda, 1} \le u_{\lambda , 2}\,$, where $u_0 < u_{\lambda_1 ,1}< u_{\lambda_2 ,1} $ if $\,0<\lambda_1<\lambda_2\,$,
$u_{\lambda , 1} \rightarrow u_0$ in $C^1(\overline{\Omega})$, and $\max_{\overline{\Omega}} u_{\lambda , 2} \rightarrow +\infty$ as ${\lambda\rightarrow 0^+}$.
The problem $(P_{\bar{\lambda}})$ has a unique strong solution; and for $\lambda > \bar{\lambda}$, the problem \eqref{Plambda} has no nonnegative solution.
\end{prop}

\begin{figure}[!htb]
\centering
\includegraphics[scale=0.18]{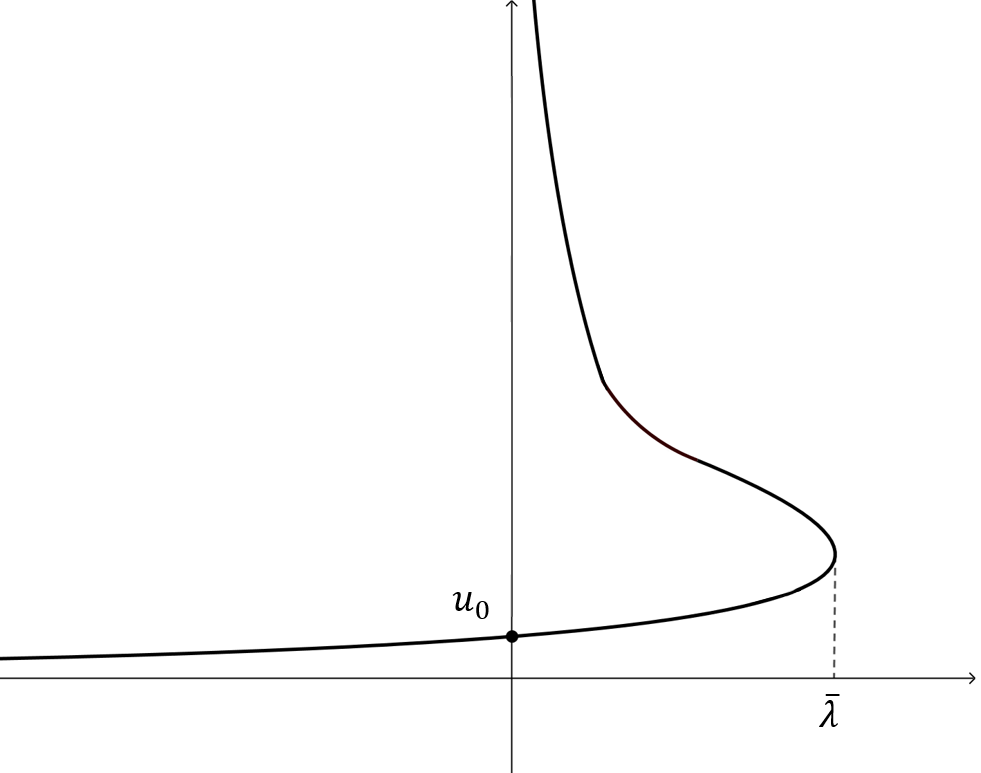}
\caption{ \small{Illustration of Proposition \ref{th1.4}.}}
\label{Rotulo2}
\end{figure}

Thus, the continuum of solutions illustrated in Figure 1 consists of radial and strictly decreasing functions by Theorem \ref{Tvisc}, provided $c,\mu,h$ are nonincreasing with the radius. 

\medskip

\textbf{{Acknowledgment.}} Ederson Moreira dos Santos was partially supported by CNPq grant 307358/2015-1, and Gabrielle Nornberg was supported by FAPESP grant 2018/04000-9.

\end{document}